\documentclass[11pt,twoside]{amsart}
\usepackage{latexsym,amssymb,amsmath}

\numberwithin{equation}{section}
\newtheorem{theorem}{Theorem}[section]
\newtheorem{lemma}[theorem]{Lemma}
\newtheorem{proposition}[theorem]{Proposition}

\theoremstyle{definition}
\newtheorem{definition}[theorem]{Definition}
\newtheorem{remark}[theorem]{Remark}
\newtheorem{example}[theorem]{Example}

\begin{document}


\newcommand{\m}[1]{\marginpar{\addtolength{\baselineskip}{-3pt}{\footnotesize
\it #1}}}
\newcommand{\A}{\mathcal{A}}
\newcommand{\K}{\mathcal{K}}
\newcommand{\knd}{\mathcal{K}^{[d]}_n}
\newcommand{\F}{\mathcal{F}}
\newcommand{\N}{\mathbb{N}}
\newcommand{\pr}{\mathbb{P}}
\newcommand{\Z}{\mathbb{Z}}
\newcommand{\R}{\mathbb{R}}
\newcommand{\I}{\mathit{I}}
\newcommand{\G}{\mathcal{G}}
\newcommand{\D}{\mathcal{D}}
\newcommand{\x}{\mathbf{x}}
\newcommand{\lcm}{\operatorname{lcm}}
\newcommand{\ndp}{N_{d,p}}
\newcommand{\tor}{\operatorname{Tor}}
\newcommand{\reg}{\operatorname{reg}}
\newcommand{\mf}{\mathfrak{m}}

\def\bb{{{\rm \bf b}}}
\def\cc{{{\rm \bf c}}}


\title{Combinatorics of Cremona monomial maps}

\author{Aron Simis}
\address{
Departamento de Matem\'atica\\
Universidade Federal
de Pernambuco\\
50740-540 Recife\\ Pe\\ Brazil
}\email{aron@dmat.ufpe.br}
\thanks{The first author was partially supported by a grant
of CNPq. He warmly thanks CINVESTAV for support during a visit.
The second author was partially supported by CONACyT
grant 49251-F and SNI}
\author{Rafael H. Villarreal}
\address{
Departamento de
Matem\'aticas\\
Centro de Investigaci\'on y de Estudios
Avanzados del
IPN\\
Apartado Postal
14--740 \\
07000 Mexico City, D.F.
}
\email{vila@math.cinvestav.mx}

\keywords{} \subjclass[2000]{14E05, 14E07, 15A51, 15A36}

\begin{abstract}
We study Cremona monomial maps using linear algebra, lattice theory 
and linear optimization methods. Among  the results is
a simple integer matrix theoretic proof  that the inverse  of a
Cremona monomial map 
is also defined by monomials
of fixed degree, and moreover, the set of monomials
defining the inverse can be obtained explicitly in terms of the
initial data. We present another method to compute the inverse of a
Cremona monomial map based on integer programming
techniques and the notion of a Hilbert basis. 
A neat consequence is drawn for the plane Cremona monomial
group, in particular the known result saying that a  plane Cremona
monomial map 
and its inverse have the same degree. 
\end{abstract}

\maketitle

\section{Introduction}

The expression ``birational combinatorics'' has been introduced in
\cite{birational-linear} 
to mean the combinatorial theory of
rational maps $\pr^{n-1}\dasharrow \pr^{m-1}$ defined by
monomials, along with natural integer arithmetic criteria for such
maps to be birational onto their 
image varieties. As claimed there, both the theory and the criteria
were intended to be 
a simple transcription of the initial geometric data.
Yet another goal is to write characteristic-free results. Thus, here
too one works over an arbitrary field 
in order that the theory be essentially independent of the nature of
the field of coefficients, 
specially when dealing with square-free monomials.

In this paper, we stick to the case where $m=n$ and deal with Cremona
maps. 
An important step  has been silently taken for granted in the
background 
of \cite[Section 5.1.2]{birational-linear}, namely, that the inverse
of a Cremona monomial map 
is also defined by monomials. To be fair this result can be obtained
via the method of \cite[Section 3]{birational-linear} 
together with the criterion of \cite{bir2003}; however
the latter gives no hint on how to derive explicit data from the
given ones. The main results of this paper provide methods
for the computation of the inverse of a Cremona monomial map and the
related invariants.

Here we add a few steps to the theory, by setting up a direct way to
convert geometric results into 
numeric or combinatorial data regardless of the ground field nature. 
The conversion allows for an incursion into some of the details of
the theory of plane Cremona maps 
defined by monomials. In particular, it is shown that the group of
such maps under composition 
is completely understood without recurring to the known results about
general plane Cremona maps. 
Thus, one shows that this group is generated by two basic monomial
quadratic maps, up to reordering of variables in the source 
and the target.
The result is not a trivial consequence of Noether's theorem since
the latter requires 
composing with projective transformations, which is out of the
picture here. Moreover, the known proofs of 
Noether's theorem  (see, e.g., \cite{alberich}) reduce to various
special situations, passing through the celebrated de 
Jonqui\`eres maps which are rarely monomial.

The well-known result that a plane Cremona map and its inverse have
the same degree is shown here for 
such monomial maps by an easy numerical counting. The argument for 
general plane Cremona maps is not difficult
but requires quite a bit of geometric insight and preparation (see,
e.g., \cite[Proposition 2.1.12]{alberich}). 

Cremona monomial maps have been dealt with in \cite{pan} and in
\cite{Kor}, but the methods 
and some of the goals are different and have not been drawn upon 
here. The group structure of Cremona monomial maps is studied in
\cite{pan} by means of toric algebra with emphasis on the
multiplicative structure, 
while in the present account we take up the so to say additive side
of the problem  
by stressing the underlying integer arithmetic. In \cite{Kor} the
group of plane Cremona monomial maps is 
described by generators and relations and it is 
shown that this group (except that they allow certain coefficients) 
is generated by permutations and two degree $2$ maps: hyperbolism and 
the Steiner involution (see Section~\ref{cremona}).

\section{Integer linear algebra and Cremona inverses}

In this section we use techniques from lattice theory and integer linear
algebra to give an algorithm that computes the ``Cremona inverse
matrix'' (see Definition~\ref{cremona-inverse-def}) of a given  
$d$-stochastic square matrix $A$ with entries in $\mathbb{N}$ 
whose determinant equals $\pm d$. A matrix is called $d$-stochastic if
the sum of the entries of each column is equal to $d$. The algorithm
is included inside the proof of Theorem~\ref{ipn-ufpe}.

Recall that
if $a=(a_1,\ldots,a_n)\in {\mathbb R}^n$, its {\it
support\/} is defined as ${\rm supp}(a)=\{i\, |\, a_i\neq 0\}$.
Note that we can write $a=a^+-a^-$,
where $a^+$ and $a^-$ are two non-negative vectors
with disjoint support. The vectors $a^+$ and $a^-$ are
called the positive and negative part of $a$ respectively.
Following a familiar notation we write $|a|=a_1+\cdots+a_n$. The
$i${\it th} unit vector in $\mathbb{R}^n$ will be denoted by $e_i$. 

We begin with a lemma whose proof uses lattice theory. 

\begin{lemma}\label{starbucks-upstairs}
Let $v_1,\ldots,v_n$ be a set of vectors in
$\mathbb{N}^n$ such that $|v_i|=d\geq 1$ for all $i$ and $\det(A)=\pm
d$, where $A$ is the $n\times n$ matrix with column vectors
$v_1,\ldots,v_n$. Then $A^{-1}(e_i-e_j)\in\mathbb{Z}^n$ for all $i,j$.
\end{lemma}

\begin{proof} Fixing indices $i,j$, there are
$\lambda_1,\ldots,\lambda_n$ in $\mathbb{Q}$ 
such that $A^{-1}(e_i-e_j)=\sum_{k=1}^n\lambda_ke_k$. Notice that
$A^{-1}(e_i)$ is the $i${\it th} column of $A^{-1}$. Set
$\mathbf{1}=(1,\ldots,1)$. Since $\mathbf{1}A=d\mathbf{1}$,
we get $\mathbf{1}/d=\mathbf{1}A^{-1}$. Therefore
$|A^{-1}(e_i)|=|A^{-1}(e_j)|=1/d$ and $\sum_k\lambda_k=0$. Then we can
write
$$
A^{-1}(e_i-e_j)=\sum_{k=2}^n\lambda_k(e_i-e_1)\ \Longrightarrow\
e_i-e_j=\sum_{k=2}^n\lambda_k(v_k-v_1).
$$
Thus there is $0\neq s\in\mathbb{N}$ such that
$s(e_i-e_j)$ belong to $\mathbb{Z}\{v_1-v_k\}_{k=2}^n$, the subgroup
of $\mathbb{Z}^n$ generated by $\{v_1-v_k\}_{k=2}^n$. By
\cite[Lemma 2.2 and Theorem 2.6]{birational-linear}, the quotient group
$\mathbb{Z}^n/\mathbb{Z}\{v_1-v_k\}_{k=2}^n$ is free, in particular
has no non-zero torsion elements. Then we 
can write
$$
e_i-e_j=\eta_2(v_2-v_1)+\cdots+\eta_n(v_n-v_1),
$$
for some $\eta_i$'s in $\mathbb{Z}$. Since
$\mathbb{Z}\{v_1-v_k\}_{k=2}^n$ is also free (of rank $n-1$), 
the vectors $v_2-v_1,\ldots,v_n-v_1$
are linearly independent. Thus
$\lambda_k=\eta_k\in \mathbb{Z}$ for all $k\geq 2$,
hence ultimately $A^{-1}(e_i-e_j)\in\mathbb{Z}^n$.
\end{proof}

Next we state our main result of integer linear algebra nature.
Its geometric translation to Cremona monomial maps and applications
will be given in 
Section~\ref{cremona}. 

\begin{theorem}\label{ipn-ufpe} Let $v_1,\ldots,v_n$ be a set of
vectors in 
$\mathbb{N}^n$ such that $|v_i|=d\geq 1$ for all $i$ and $\det(A)=\pm
d$, where $A$ is the $n\times n$ matrix with column vectors
$v_1,\ldots,v_n$. Then there are unique vectors
$\beta_1,\ldots,\beta_n,\gamma \in \mathbb{N}^n$ such that the
following two conditions hold{\rm :}
\begin{enumerate}
\item[(a)] $A\beta_i=\gamma+e_i$ for all $i$, where $\beta_i,\gamma$
and $e_i$ are regarded as column vectors$\,${\rm ;}
\item[(b)] The matrix $B$ whose columns are
$\beta_1,\ldots,\beta_n$ has at least one zero entry in every row.
\end{enumerate}
Moreover, $\det(B)=\pm (|\gamma|+1)/d=\pm |\beta_i|$ for all $i$.
\end{theorem}

\begin{proof} First we show the uniqueness. Assume that
$\beta_1',\ldots,\beta_n',\gamma'$ is a set of vectors in
$\mathbb{N}^n$ such that: (a') $A\beta_i'=\gamma'+e_i$ for all $i$,
and (b') The matrix $B'$ whose column vectors are
$\beta_1',\ldots,\beta_n'$ has at least one zero entry in every row.
Let $\Delta=(\Delta_i)$ and $\Delta'=(\Delta_i')$ be non-negative
vectors such that $A^{-1}(\gamma-\gamma')=\Delta'-\Delta$. Then from
(a) and (a') we get
\begin{equation}\label{starbucks}
\beta_i-\beta_i'=A^{-1}(\gamma-\gamma')=\Delta'-\Delta,\ \forall\,
i\ \Longrightarrow\ \beta_{ik}-\beta_{ik}'=\Delta_k'-\Delta_k,\
\forall\, i,k,
\end{equation}
where $\beta_i=(\beta_{i1},\ldots,\beta_{in})$ and
$\beta_i'=(\beta_{i1}',\ldots,\beta_{in}')$. It suffices to show that
$\Delta=\Delta'$. If $\Delta_k'>\Delta_k$ for some $k$,
then, by Eq.~(\ref{starbucks}), we obtain $\beta_{ik}>0$ for
$i=1,\ldots,n$, which contradicts (b). Similarly if
$\Delta_k'<\Delta_k$ for some $k$, 
then, by Eq.~(\ref{starbucks}), we obtain $\beta_{ik}'>0$ for
$i=1,\ldots,n$, which contradicts (b'). Thus $\Delta_k=\Delta_k'$ for
all $k$, i.e., $\Delta=\Delta'$.

Next we prove the existence of $\beta_1,\ldots,\beta_n$ and $\gamma$.
By Lemma~\ref{starbucks-upstairs}, for $i\geq 2$ we can write
$$
0\neq\alpha_i=A^{-1}(e_1-e_i)=\alpha_i^+-\alpha_i^-
$$
where $\alpha_i^+$ and $\alpha_i^-$ are in $\mathbb{N}^n$. Notice
that $\alpha_i^+\neq 0$ and $\alpha_i^-\neq 0$. Indeed
the sum of the entries of $A^{-1}(e_i)$ is equal to $1/d$.
Thus $|\alpha_i|=|\alpha_i^+|-|\alpha_i^-|=0$, and consequently the
positive and 
negative part of $\alpha_i$ are both non-zero for $i\geq 2$. The
vector $\alpha_i^+$ can be written as
$\alpha_i^+=(\alpha_{i1}^+,\ldots,\alpha_{in}^+)$ for $i\geq 2$. For
$1\leq k\leq n$ consider the integers given by
$$
m_k=\max_{2\leq i\leq n}\{\alpha_{ik}^+\}
$$
and set $\beta_1=(m_1,\ldots,m_n)$. Since $\beta_1\geq\alpha_i^+$, for
each $i\geq 2$ there is $\theta_i\in\mathbb{N}^n$ such that
$\beta_1=\theta_i+\alpha_i^+$. Therefore
$$
\alpha_i=A^{-1}(e_1-e_i)=\alpha_i^+-\alpha_i^-=\beta_1-(\theta_i+\alpha_i^-).
$$
We set $\beta_i=\theta_i+\alpha_i^-$ for $i\geq 2$. Since we
have $A\beta_1-e_1=A\beta_i-e_i$ for $i\geq 2$, it follows readily
that $A\beta_1-e_1\geq 0$ (make $i=2$ in the last equality and
compare entries).
Thus,
setting $\gamma:=A\beta_1-e_1$, it follows that $\beta_1,\ldots,\beta_n$
and $\gamma$ satisfy (a). If each row of $B$ has some zero entry the
proof of the existence is complete. If every entry of a row of $B$ is
positive we 
subtract the vector $\mathbf{1}=(1,\ldots,1)$ from that row and change
 $\gamma$ accordingly so that (a) is still satisfied. Applying this
argument repeatedly we get a sequence $\beta_1,\ldots,\beta_n,\gamma$
satisfying (a) and (b).

We now prove the last part of the assertion. Notice that if 
$\beta_{ij}$ denotes the $j$-entry of $\beta_i$, then the equality 
$A\beta_i=\gamma+e_i$ is equivalent to
$\beta_{i1}v_1+\cdots+\beta_{in}v_n=\gamma+e_i$. Thus
$|\beta_i|d=|\gamma|+1$. Note that condition (a) is equivalent to the
equality 
$AB=\Gamma+I$, where $\Gamma$ is the matrix all of whose columns are
equal to $\gamma$. Since $\det(B)=\pm \det (\Gamma+I)/d$
it suffices to show that
$\det(\Gamma+I)=|\gamma|+1$.
The latter is a classical calculation than can be performed in
various ways. One has a more general statement which is given below in
Lemma~\ref{Gordan}. In particular by this lemma, 
if $\Gamma$ has rank at most one and $D$ is the identity matrix, one gets
$\det(\Gamma+I)={\rm trace}(\Gamma)+1$, as required.  This completes
the proof of the theorem.  
\end{proof}

\begin{lemma}\label{Gordan}
Let $\Gamma=(\gamma_{i,j})$ be an $n\times n$  square matrix
over an arbitrary commutative ring, and let
$D={\rm diag}(d_1,\ldots, d_n)$ be a diagonal matrix over the same ring.
Then
\begin{eqnarray*}
\det (\Gamma+D)&=&\det(\Gamma) \\
&+&\sum_id_i \Delta _{[n]\setminus \{i\}}+
\sum_{1\leq i_1<i_2\leq n}d_{i_1}d_{i_2} \Delta_{[n]\setminus\{i_1,i_2\}} \\
&+&\cdots +
\sum_{1\leq i_1<\cdots <i_{n-1}\leq n}d_{i_1}\cdots d_{i_{n-1}}
\Delta_{[n]\setminus\{i_1,\ldots,i_{n-1}\}} \\ 
& +& \det D,
\end{eqnarray*}
where $[n]=\{1,\ldots, n\}$ and $\Delta_{[n]\setminus\{i_1,\ldots,i_{k}\}}$
denotes the principal $(n-k)\times (n-k)$-minor of $\Gamma$ with rows
and columns 
$[n]\setminus \{i_1,\ldots,i_{k}\}$.
\end{lemma}

\begin{proof} It follows from the multi-linearity of the determinant. 
\end{proof}

\begin{definition}\label{cremona-inverse-def} A matrix $A$ that
satisfies the hypotheses of Theorem~\ref{ipn-ufpe} is called a {\it
Cremona matrix} and $B$ is called the {\it Cremona
inverse matrix} of $A$.
\end{definition}

\begin{remark} Bridges and Ryser \cite{ryser} (cf.
\cite[Theorem~4.4]{cornu-book}) considered a particular class of
Cremona matrices. They studied the equation $AB=\Gamma+I$, when $A,B$
are $\{0,1\}$ matrices 
and $\Gamma$ is the matrix with all its entries equal to $1$.
They show that if equality occurs, then each row and column of $A$ has
the same number $r$ of ones, each row and column of $B$ has
the same number $s$ of ones with $rs=n+1$, and $AB=BA$.
\end{remark}

As mentioned before the proof of Theorem~\ref{ipn-ufpe} provides an
algorithm to compute 
the vectors $\beta_1,\ldots,\beta_n$ and $\gamma$ (see
Example~\ref{propedeutico} and the proof of
Proposition~\ref{degree_of_inverse} for specific illustrations of the
algorithm). Other means to
compute these vectors using linear optimization techniques will be discussed in
Section~\ref{computing}.
\begin{example}\label{propedeutico}
Consider the following matrix $A$ and its inverse:
$$
A=\left(\begin{matrix}
d&d-1&0\cr
0&1&d-1\cr
0&0&1
\end{matrix}\right);\ \ \ \ A^{-1}=\frac{1}{d}\left(\begin{matrix}
1&1-d&(d-1)^2\cr
0&d&d(1-d)\cr
0&0&d
\end{matrix}\right).
$$
To compute the $\beta_i$'s and $\gamma$ we follow the proof of
Theorem~\ref{ipn-ufpe}.
Then $\beta_1=(2,d,0)$, $\beta_2=(1,d+1,0)$, $\beta_3=(d,1,1)$,
$\gamma=(d^2+d-1,d,0)$, and 
$$
B=\left(\begin{matrix}
2&1&d\cr
d&d+1&1\cr
0&0&1
\end{matrix}\right).
$$
By subtracting the vector $(1,1,1)$ from rows $1$ and $2$, we get
$$
B'=\left(\begin{matrix}
1&0&d-1\cr
d-1&d&0\cr
0&0&1
\end{matrix}\right).
$$
The column vectors $\beta_1'=(1,d-1,0)$, $\beta_2'=(0,d,0)$,
$\beta_3'=(d-1,0,1)$, $\gamma'=(d^2-d,d-1,0)$ satisfy (a) and (b).
\end{example}

\section{An integer programming method via Hilbert bases}\label{computing}

The proof of Theorem~\ref{ipn-ufpe} provides an algorithm to compute
the inverse of a Cremona matrix. In this section we present another
algorithm based 
on integer programming techniques and the notion of
a Hilbert basis.  From a complexity point of view, the first algorithm
based on simple linear algebra is far faster than the second algorithm
based on integer programming. The reason is the large number of
variables that the second method requires (for small
cases, both algorithms work fine). The second approach is quite
interesting from a theoretical point of view as shown in
\cite{costa-simis}.  

Let $v_1,\ldots,v_n$ be a set of vectors in
$\mathbb{N}^n$ such that $|v_i|=d\geq 1$ for all $i$ and $\det(A)=\pm
d$, where $A$ is the $n\times n$ matrix with column vectors
$v_1,\ldots,v_n$. Then, by Theorem~\ref{ipn-ufpe}, there are unique vectors
$\beta_1,\ldots,\beta_n,\gamma \in \mathbb{N}^n$ such that the
following two conditions hold{\rm :} (a) $A\beta_i=\gamma+e_i$ for
all $i$, and (b) The matrix $B$ whose columns are
$\beta_1,\ldots,\beta_n$ has at least one zero entry in every row.

To compute the sequence $\beta_1,\ldots,\beta_n,\gamma$ using linear
programming we regard the $\beta_i$'s and $\gamma$ as vectors of
indeterminates 
and introduce a new variable $\tau$. Consider the homogeneous
system of linear inequalities

\begin{eqnarray}\label{jan5-09}
A\beta_i&=&\gamma+\tau e_i,\ \ i=1,\ldots,n\\
\beta_i&\geq& 0,\ \ \ i=1,\ldots,n\nonumber\\
\gamma&\geq &0,\ \ \tau\geq 0.\nonumber
\end{eqnarray}
This linear system has $n^2$ equality constraints and $\ell=n^2+n+1$
indeterminates. The set $C$ of solutions form a rational pointed polyhedral
cone. By \cite[Theorem~16.4]{Schr}, there is a unique minimal
integral Hilbert basis 
$$
\mathcal{H}=\{h_1,\ldots,h_r\}
$$
of $C$ such that $\mathbb{Z}^\ell\cap
\mathbb{R}_+\mathcal{H}=\mathbb{N}\mathcal{H}$ and 
$C=\mathbb{R}_+\mathcal{H}$ (minimal relative to
taking subsets), where  $\mathbb{R}_+\mathcal{H}$ denotes the
cone generated by $\mathcal{H}$ consisting of all linear combinations of
$\mathcal{H}$ with non-negative real coefficients and
$\mathbb{N}\mathcal{H}$ denotes the semigroup generated by
$\mathcal{H}$ consisting of all linear combinations of
$\mathcal{H}$ with coefficients in $\mathbb{N}$. The Hilbert basis of
$C$ has the following useful description.

\begin{theorem}{\rm \cite[p. 233]{Schr}}\label{hb-description}
$\mathcal{H}$ is the set of all integral vectors $0\neq h\in C$ such
that $h$ is not the sum of two other non-zero integral vectors in $C$.
\end{theorem}

\begin{theorem} There is a unique element $h$ of $\mathcal{H}$ with
$\tau=1$ and this element gives the unique sequence
$\beta_1,\ldots,\beta_n,\gamma$ that satisfies $(a)$ and $(b)$.
\end{theorem}

\begin{proof} We set
$x_0=(\beta_1,\ldots,\beta_n,\gamma,1)\in\mathbb{N}^{n^2+n+1}$, where
$\beta_1,\ldots,\beta_n,\gamma$ is the unique sequence that satisfies
$(a)$ and $(b)$. First we show that $x_0$ is in $\mathcal{H}$.
Clearly $x_0$ is in $C$. Thus we may assume that
$x_0$ is written as
$$
x_0=\eta_1h_1+\cdots+\eta_kh_k,\ \ \ \ \ 0\neq
\eta_i\in\mathbb{N}\, \mbox{ for }i=1,\ldots,k,
$$
where $h_k$ has its last entry equal to $1$ and the last entry of
$h_i$ is equal to $0$ for $i<k$. The vector $h_k$ has the form
$$
h_k=(\beta_{1}^{(k)},\ldots,\beta_n^{(k)},\gamma^{(k)},1),
$$
where the $\beta_{i}^{(k)}$'s and $\gamma^{(k)}$ are in $\mathbb{N}^n$
and satisfy Eq.~(\ref{jan5-09}), i.e., they satisfy (a). Notice that from
the first equality one has $x_0\geq h_k$. Then $\beta_i\geq
\beta_i^{(k)}$ for all $i$ and $\gamma\geq \gamma^{(k)}$. Therefore
the $\beta_{i}^{(k)}$'s and $\gamma^{(k)}$ also satisfy (b).
Consequently $x_0=h_k$ and $x_0\in\mathcal{H}$, as claimed.
Let $h_i$ be any other element in $\mathcal{H}$ whose last entry
is equal to $1$. Next we show that $h_i$ must be equal to $x_0$. The
vector $h_i$ has the form
$$
h_i=(\beta_{1}^{(i)},\ldots,\beta_n^{(i)},\gamma^{(i)},1),
$$
where the $\beta_{j}^{(i)}$'s and $\gamma^{(i)}$ are in $\mathbb{N}^n$
and satisfy Eq.~(\ref{jan5-09}). Since the $\beta_{j}^{(i)}$'s and
$\gamma^{(i)}$ satisfy (a), it suffices to show that they also
satisfy (b). We proceed by contradiction. Assume that the
$\ell${\it th} entry of $\beta_{j}^{(i)}$ is not zero for
$j=1,\ldots,n$. For simplicity of notation assume that $\ell=1$. Then
the vectors
$$
h'=(e_1,\ldots,e_1,Ae_1,0)\ \mbox{and }\
h''=(\beta_{1}^{(i)}-e_1,\ldots,\beta_n^{(i)}-e_1,\gamma^{(i)}-Ae_1,1)
$$
are integral vectors that satisfy Eq.~(\ref{jan5-09}), i.e.,
$h'$ and $h''$ are integral vectors in $C$ and $h=h'+h''$, a
contradiction to Theorem~\ref{hb-description}.
\end{proof}

There
are computer programs that can be used to find integral Hilbert
basis of polyhedral cones defined by linear systems of the form
$x\geq 0\, ;A'x=0$, where $A'$ is an integral matrix. We have used
\cite{normaliz2} to compute some specific examples with this
procedure:

\begin{example}\label{propedeutico-d=2}
Consider the following matrix
$$
A=\left(\begin{matrix}
2&1&0\cr
0&1&1\cr
0&0&1
\end{matrix}\right)
$$
To compute $\beta_1,\beta_2,\beta_3,\gamma$ we use the following
input file for {\it Normaliz\/} \cite{normaliz2}
\begin{small}
\begin{verbatim}
9
13
2 1 0  0 0 0  0 0 0  -1 0 0  -1
0 1 1  0 0 0  0 0 0  0 -1 0  0
0 0 1  0 0 0  0 0 0  0 0 -1  0
0 0 0  2 1 0  0 0 0  -1 0 0  0
0 0 0  0 1 1  0 0 0  0 -1 0  -1
0 0 0  0 0 1  0 0 0  0 0 -1  0
0 0 0  0 0 0  2 1 0  -1 0 0  0
0 0 0  0 0 0  0 1 1  0 -1 0  0
0 0 0  0 0 0  0 0 1  0 0 -1  -1
5
\end{verbatim}
\end{small}
Part of the output file produced by {\it Normaliz} is:
\begin{small}
\begin{verbatim}
4 generators of integral closure:
 0 1 0 0 1 0 0 1 0 1 1 0 0
 0 0 1 0 0 1 0 0 1 0 1 1 0
 1 0 0 1 0 0 1 0 0 2 0 0 0
 1 1 0 0 2 0 1 0 1 2 1 0 1
\end{verbatim}
\end{small}
From the last row we get:
$\beta_1=(1,1,0),\beta_2=(0,2,0),\beta_3=(1,0,1),\gamma=(2,1,0)$.
\end{example}

\section{Application to Cremona maps}\label{cremona}

Let $R=k[x_1,\ldots,x_n]$ be a polynomial ring over a field $k$. 
Given $\alpha=(a_1,\ldots,a_n)\in {\mathbb N}^n$, we set 
${x}^{\alpha}:= x_1^{a_1}\cdots x_{n}^{a_{n}}$. In what follows we
consider a finite set of distinct monomials ${F}=\{{x}^{v_1},\ldots,
{x}^{v_n}\}\subset R$ of the same degree $d\geq
1$ and having no non-trivial common factor. We also assume throughout that every $x_i$
divides at least one member of $F$, a harmless condition.
The set $F$ defines a rational (monomial) map $\pr^{n-1}\dasharrow \pr^{n-1}$
which will also be denoted $F$ and written as a tuple
$F=({x}^{v_1},\ldots, {x}^{v_n})$. This map is said to be a {\em
Cremona map\/} (or a {\em Cremona transformation\/}) if it admits an
inverse rational map with source $\pr^{n-1}$. Note that a rational
monomial map is defined 
everywhere if and only if the defining monomials are pure powers of the
variables, in which case it is a Cremona map if and only if $d=1$
(the identity map up to permutation of variables or more precisely a
Cremona map of the form $F=(x_{\sigma(1)},\ldots,x_{\sigma(n)})$ for
some $\sigma$ in the symmetric group $S_n$). 
Finally, the integer $d$ is often called {\em
degree\/} of $F$ (not to be confused with its
degree as a map).

The {\it log-matrix\/} of $F$, denoted by $A$, is the $n\times n$ matrix with
column vectors $v_1,\ldots,v_n$. The next result will be used in
several places. 

\begin{proposition}{\rm\cite[Proposition
2.1]{SiVi}}\label{march27-11} $F$ is a Cremona 
map if and only if $\det(A)=\pm d$.
\end{proposition} 

\begin{theorem}\label{transcription} Let $F:\pr^{n-1}\dasharrow
\pr^{n-1}$ stand for a rational map defined by monomials
of fixed degree.
If $F$ is a Cremona map then its inverse  is also defined by monomials
of fixed degree.
Moreover, the degree as well as a set of monomials defining the
inverse can be obtained explicitly in terms of the given set of monomials 
defining $F$.
\end{theorem}
\begin{proof} Let $F$ be a Cremona map defined by a set of monomials
$f_1,\ldots, f_n$ of the same degree $d\geq 1$. 
We set $f_i=x^{v_i}$ for $i=1,\ldots,n$. 
By Proposition~\ref{march27-11} the matrix of exponents of these
monomials (i.e., their log-matrix) $A$ has determinant $\pm d$.
Therefore Theorem~\ref{ipn-ufpe} implies the existence of an $n\times
n$ matrix $B$ such that $AB=\Gamma+I$, 
where $\Gamma$ is a matrix with repeated column $\gamma$ throughout.
Let $g_1,\ldots, g_n$ denote the monomials whose log-matrix is
$B$ and call $G$ the corresponding rational monomial map.  Letting
$x^{\gamma}$ denote
the monomial whose exponents are the coordinates of $\gamma$, the
above matrix equality translates into the equality 
$$
(f_1(g_1,\ldots, g_n),\ldots ,f_n(g_1,\ldots, g_n))=
({x}^{\gamma}\cdot x_1,\ldots, {x}^{\gamma}\cdot x_n).
$$
Thus the left hand side is proportional to the vector
$(x_1,\ldots,x_n)$ which means that the composite map $F\circ G$ 
is the identity map wherever the two are defined
(see \cite[proof of Proposition 2.1]{bir2003}).
On the other hand, since $B$ is the log-matrix of $g_1,\ldots, g_n$,
Theorem~\ref{ipn-ufpe} applied in the opposite
direction says that $G$ is also a Cremona map. Therefore $G$ has to
be the inverse of $F$, as required. Finally, notice that the proof of
Theorem~\ref{ipn-ufpe} provides an
algorithm to compute $B$ and $\gamma$. The input for this algorithm is
the log-matrix $A$ of $f_1,\ldots, f_n$.
\end{proof}

We will call a Cremona map as above a {\em Cremona monomial map}.
The theorem allows us to introduce the following group.

\begin{definition}\label{monomialgroup}
The {\em Cremona monomial group} of order $n-1$ is the subgroup
of the Cremona group of $\pr^{n-1}$ whose elements are Cremona monomial maps.
\end{definition}
Here we will not distinguish between rational maps $F,G:\pr^{n-1}\dasharrow \pr^{m-1}$
defined, respectively, by forms $f_1,\ldots, f_m$ of the same degree
and their multiples $g_1=gf_1,\ldots, g_m=gf_m$ by a fixed form $g$ of arbitrary degree.

There is a potential confusion in this terminology. For instance, in \cite{Kor} one
allows for more general maps by considering the free product with certain Klein groups.
Since our goal is combinatorial we will stick to the above definition
of the Cremona monomial group.

As a matter of notation, the composite of two Cremona maps $F,G$ will
be indicated by $FG$
(first $G$, then $F$).
Likewise, the power composite $F\cdots F$ of $m$ factors will be denoted $F^m$.

We shift our attention to {\em plane} Cremona monomial maps, i.e., we will study
the structure of the Cremona monomial group for $n=3$. Consider the maps
$H=(x_1^2,x_1x_2,x_2x_3)$ and $S=(x_1x_2,x_1x_3,x_2x_3)$.

$\bullet$  For any $d\geq 1$, one has $H^{d-1}=(x_1^d, x_1^{d-1}x_2, x_2^{d-1}x_3)$.

This is a straightforward composite calculation: by induction, we
assume that $H^{d-2}$ is defined by $x_1^{d-1}, x_1^{d-2}x_2, x_2^{d-2}x_3$, 
hence $H^{d-1}=H^{d-2}H$ is defined by $x_1^{2(d-1)}, x_1^{2d-3}x_2,
x_1^{d-2}x_2^{d-1}x_3$ which is the same rational map as 
the one defined by $x_1^d, x_1^{d-1}x_2, x_2^{d-1}x_3$ (by canceling
the $\gcd$ $x_1^{d-2}$)

$\bullet$  For any $d\geq 1$, up to permutation of the variables both in
the source and
the target, $H^{d-1}$ is
involutive, i.e., coincides with its own inverse -- in
\cite{birational-linear} this was called ``\/$p$-involutive''.

Namely, consider the Cremona  map
$G=(x_1x_2^{d-1},x_2^d,x_1^{d-1}x_3)$. Again, one finds
straightforwardly that  
$H^{d-1}G$ is defined by $x_1^dx_2^{d(d-1)},
x_1^{d-1}x_2^{(d-1)^2+d}x_2^d, x_1^{d-1}x_2^{d(d-1)}x_3$; canceling
the $\gcd$ $x_1^{d-1}x_2^{d(d-1)}$
yields the identity map $(x_1,x_2,x_3)$.

$\bullet$  The monomial maps $S$ and $H$ ``nearly'' commute;
actually they are conjugate by a transposition.

This is once more straightforward: the general composites $SH^{d-1}$
and $H^{d-1}S$ are defined, respectively, by 
$x_1^d, x_1x_2^{d-2}x_3, x_2^{d-1}x_3$ and $x_1x_2^{d-1},
x_1x_2^{d-2}x_3, x_3^d$ and these are the same up to transposing
$x_1$ and $x_3$ and  
the extreme terms. This means that they are conjugate by a transposition.
 It has as a consequence that the subgroup of the Cremona group
of $\pr^2$ generated by $S$ and $H$ is Abelian up to free product
with the symmetric group $S_3$ (cf. \cite[p.~1672]{Kor}). 

\begin{lemma}\label{pillo-boda-28march}
Let $F$ be a plane Cremona map of degree $d$ of the form
$F=(x_1^{a_1}x_2^{a_2},x_2^{b_2}x_3^{b_3},x_1^{c_1}x_3^{c_3})$. Then,
up to permutation of the source $($variables$)$ and the target
$($monomials$)$, $F$ is one of the following
two kinds:
$$F=(x_1x_2,x_2x_3,x_1x_3) \ \mbox{ or }\
F=(x_1^d,x_2x_3^{d-1},x_1^{d-1}x_3).
$$
\end{lemma}

\begin{proof} Let $A$ be the log-matrix of $F$. The determinant of $A$
is equal to $\pm d$ by Proposition~\ref{march27-11}. Since
$\det(A)=a_1b_2c_3+c_1a_2b_3$, we get that $\det(A)=d$. Hence, one has 
\begin{equation}\label{march23-09}
d=a_1+a_2=b_2+b_3=c_1+c_3=a_1b_2c_3+c_1a_2b_3\ \mbox{ and }\
a_1(b_2c_3-1)=a_2(1-c_1b_3). 
\end{equation}
The cases below can be readily verified using these equations.

Case (I): $a_1\geq 1$, $a_2=0$. Then $a_1=d$, $b_2=c_3=1$,
and $F=(x_1^d,x_2x_3^{d-1},x_1^{d-1}x_3)$.

Case (II): $a_1=0$, $a_2\geq 1$. Then $a_2=d$, $b_3=c_1=1$,
and $F=(x_2^d,x_2^{d-1}x_3,x_1x_3^{d-1})$.

Case (III)(a): $a_1\geq 1$, $a_2\geq 1$, $b_2c_3=0$, $b_2=0$. Then
$F=(x_1^{d-1}x_2,x_3^d,x_1x_3^{d-1})$.

Case (III)(b): $a_1\geq 1$, $a_2\geq 1$, $b_2c_3=0$, $c_3=0$. Then
$F=(x_1^{d-1}x_2,x_2^{d-1}x_3,x_1^d)$.

Case (III)(c): $a_1\geq 1$, $a_2\geq 1$, $c_1b_3=0$, $c_1=0$. Then
$F=(x_1x_2^{d-1},x_2x_3^{d-1},x_3^d)$.

Case (III)(d): $a_1\geq 1$, $a_2\geq 1$, $c_1b_3=0$, $b_3=0$. Then
$F=(x_1x_2^{d-1},x_2^{d},x_1^{d-1}x_3)$.

Case (III)(e): $a_1\geq 1$, $a_2\geq 1$, $b_2c_3-1\geq 0$,
$1-c_1b_3\leq 0$.  From the last equality in Eq.~(\ref{march23-09}) we get
$b_2c_3=1$ and $c_1b_3=1$. Then $F=(x_1x_2,x_2x_3,x_1x_3)$.
\end{proof}

We next give a purely integer matrix theoretic proof of the following
result which was essentially proved in \cite[Section~3]{Kor}. 

\begin{proposition}\label{group_structure}
If $n=3$, then up to permutation of the variables
{\rm (}both in the source and the target{\rm )},
the plane Cremona monomial group is generated by the maps $S, H$,
where $S$ is the quadratic map of first kind defined
by $x_1x_2, x_1x_3, x_2x_3$ {\rm (}Steiner involution{\rm )}
and $H$ is the quadratic map of second kind
defined by $x_1^2, x_1x_2, x_2x_3$ {\rm (}``hyperbolism''{\rm )}.
\end{proposition}
\begin{proof} Let
$F=(x_1^{a_1}x_2^{a_2}x_3^{a_3},x_1^{b_1}x_2^{b_2}x_3^{b_3},x_1^{c_1}x_2^{c_2}x_3^{c_3})$
be a plane Cremona map of degree $d$.
We will prove that, up to a permutation of the
variables and the monomials, every plane Cremona monomial map is of the form
$F\cdots SH^{d_{i_k}}SH^{d_{i_{k+1}}}\cdots G$,
where $F,G$ are in $\{S, H^{d}\,|\, d\geq 1\}$. The proof is by
induction on the degree. Consider the log-matrix of $F$:
$$
A=\left(\begin{matrix} a_1&b_1&c_1\\
a_2&b_2&c_2\\
a_3&b_3&c_3
\end{matrix}
\right).
$$
Up to permutation of variables and monomials, there are
essentially two cases to consider:
$$
\begin{array}{ccc}
A=\left(\begin{matrix} a_1&b_1&0\\
a_2&b_2&0\\
a_3&0&d
\end{matrix}
\right)
&
\ \ \ \mbox{ or }\ \ \
&
A=\left(\begin{matrix} a_1&0&c_1\\
a_2&b_2&0\\
0&b_3&c_3
\end{matrix}
\right).
\end{array}
$$
The determinant of $A$ is $\pm d$ by Proposition~\ref{march27-11}. 
By Lemma~\ref{pillo-boda-28march} we may assume that $A$ is the matrix
on the left, i.e.,
$F=(x_1^{a_1}x_2^{a_2}x_3^{a_3},x_1^{b_1}x_2^{b_2},x_3^d)$.

Case (I): $a_1\geq b_1$. From $a_1+a_2+a_3=b_1+b_2=d$, we get that
$b_2\geq a_2$. Consider the Steiner involution given by
$S=(x_1x_2,x_2x_3,x_1x_3)$. Then
$$
SF=(x_1^{a_1+b_1}x_2^{a_2+b_2}x_3^{a_3},
x_1^{b_1}x_2^{b_2}x_3^{d},
x_1^{a_1}x_2^{a_2}x_3^{d+a_3})=
x_1^{b_1}x_2^{a_2}x_3^{a_3}(x_1^{a_1}x_2^{b_2},x_2^{b_2-a_2}x_3^{d-a_3},
x_1^{a_1-b_1}x_3^{d}).
$$
Thus by Lemma~\ref{pillo-boda-28march} we get that $SF=H^{a_1+b_2-1}$
for some hyperbolism $H$ or
$SF$ is a Steiner involution. Thus multiplying by the inverse of $S$,
we obtain that $F$ has the required form.

Case (II)(a): $b_1>a_1$ and $\det(A)=d$. By
Lemma~\ref{pillo-boda-28march} we may assume $a_1\geq 1$. From the
equality
$d=\det(A)=d(a_1b_2-a_2b_1)$, we get that $b_2\geq a_2$. Consider the
hyperbolism $H=(x_1^2,x_1x_2,x_2x_3)$.
In this case we have
\begin{eqnarray*}
HF&=&(x_1^{2a_1}x_2^{2a_2}x_3^{2a_3},
x_1^{a_1+b_1}x_2^{a_2+b_2}x_3^{a_3},
x_1^{b_1}x_2^{b_2}x_3^{d})\\
&=&
x_1^{a_1+1}x_2^{a_2}x_3^{a_3}
(x_1^{a_1-1}x_2^{a_2}x_3^{a_3},x_1^{b_1-1}x_2^{b_2},
x_1^{b_1-a_1-1}x_2^{b_2-a_2}x_3^{d-a_3})\\
&=&x_1^{a_1+1}x_2^{a_2}x_3^{a_3}F_1.
\end{eqnarray*}
Since $F_1$ has degree at most $d-1$,
we have lowered the degree of $F$. Thus by induction $F$ has the required
form.

Case (II)(b): $b_1>a_1$ and $\det(A)=-d$. We may assume that
$a_2\geq b_2$, otherwise we may proceed as in Case (II)(a). By
Lemma~\ref{pillo-boda-28march} we may also assume that $a_1\geq 1$. We
claim that $F$ has the form $F=(x_1^{d-2}x_2x_3,x_1^{d-1}x_2,x_3^d)$.
The condition on the determinant is equivalent to the equality
$a_2b_1-a_1b_2=1$. Thus $a_2\geq 1$. Hence using that $b_1\geq a_1+1$
one has
$$
a_2b_1\geq a_2a_1+a_2\geq b_2a_1+1=a_2b_1.
$$
Consequently $a_2=b_2=1$. The condition $\det(A)=-d$, becomes
$a_1=b_1-1=d-2$. Therefore $F$ has the asserted form. Consider the
hyperbolism $H=(x_1^2,x_1x_2,x_2x_3)$. It is easy to see that
$HF$ is equal to $x^{\gamma}(x_1^{d-3}x_2x_3,x_1^{d-2}x_2,x_3^{d-1})$
for some monomial $x^{\gamma}$. Thus we have lowered the degree of $F$
and we may apply induction.
\end{proof}

The next result is classically well-known. We give a simple direct
proof in the case 
of plane Cremona monomial maps. The  proof shows moreover
explicit simple formulae for the Cremona inverse in the case of $3$ variables.

\begin{proposition}\label{degree_of_inverse} If $n=3$, then a plane Cremona
monomial map and its inverse have the same degree.
\end{proposition}

\begin{proof} The proof is based on the method employed in the
proof of Theorem~\ref{ipn-ufpe}. Let $A$ be the log-matrix of a plane
Cremona map $F$ of degree $d$:
$$
A=\left(\begin{matrix} a_1&b_1&c_1\\
a_2&b_2&c_2\\
a_3&b_3&c_3
\end{matrix}
\right).
$$
By Proposition~\ref{march27-11} we may assume that $\det(A)=d$, the case $\det(A)=-d$
can be shown similarly. Up to permutation of variables, as in the proof of
Proposition~\ref{group_structure}, there are
essentially two cases to consider:
$$
\begin{array}{ccc}
A=\left(\begin{matrix} a_1&b_1&0\\
a_2&b_2&0\\
a_3&0&d
\end{matrix}
\right)
&
\ \ \ \mbox{ or }\ \ \
&
A=\left(\begin{matrix} a_1&0&c_1\\
a_2&b_2&0\\
0&b_3&c_3
\end{matrix}
\right).
\end{array}
$$
It is readily seen that the inverse of $A$ is given by
$$
\begin{array}{ccc}
A^{-1}=\left(\begin{matrix} b_2&-b_1&0\\
-a_2&a_1&0\\
-a_3b_2/d&b_1a_3/d&1/d
\end{matrix}
\right)
&
\ \ \ \mbox{ or }\ \ \
&
A^{-1}=\displaystyle\frac{1}{d}\left(\begin{matrix}
b_2 c_3& b_3 c_1& -b_2 c_1\\
-a_2 c_3& a_1 c_3& a_2 c_1\\
a_2
b_3& -a_1 b_3& a_1 b_2
\end{matrix}
\right).
\end{array}
$$
respectively. By the argument given in the proof of
Theorem~\ref{ipn-ufpe} we get that
$$
\beta_1=(d,0,0),\  \gamma=(da_1-1,da_2,da_3)\ \ \mbox{ or }\ \
\beta_1=(b_2,0,b_3),\  \gamma=(a_1 b_2 + b_3 c_1-1, a_2 b_2,
b_3 c_3)
$$
respectively. Therefore
$$
\begin{array}{ccc}
B=\left(\begin{matrix} d&0&b_1\\
0&a_1+a_2&a_2\\
0&a_3&(a_3b_2+1)/d
\end{matrix}
\right)
&
\ \ \ \mbox{ or }\ \ \
&
B=\left(\begin{matrix} b_2&c_1&0\\
0&c_3&a_2\\
b_3&0&a_1
\end{matrix}
\right).
\end{array}
$$
respectively. Using that the sum of the entries in every column of $A$ is equal to
$d$ and $\det(A)=d$, it is seen that the sum of the entries in each column
of $B$ is equal to $d$ and $\det(B)=d$.\end{proof}

One of the peculiarities of the theory is that even if the given
monomials are square-free
to start with, the inverse map is generally defined by non-square-free monomials. This
makes classification in high degrees, if not the structure of the
Cremona monomial group
itself, a difficult task (see \cite{costa-simis,pan}). 
A complete classification of monomial
Cremona maps of degree $2$ in any number of variables is given in \cite{costa-simis}. 

\bigskip

\noindent{\bf Acknowledgment.}  The authors would like to thank an 
anonymous referee for providing us 
with useful comments and suggestions. 

\bibliographystyle{plain}

\begin{thebibliography}{99}

\bibitem{alberich}{M. Alberich-Carrami\~nana, {\it Geometry of the
Plane Cremona Maps}, Lecture Notes in Mathematics, vol. 1769, 2002,
Springer-Verlag, Berlin-Heidelberg.}

\bibitem{ryser} W. G. Bridges and H. J. Ryser,
Combinatorial designs and related systems,  J. Algebra  {\bf 13}
(1969),  432--446.

\bibitem{normaliz2} W. Bruns and B. Ichim, \textsc{Normaliz 2.0},
Computing
normalizations of affine
semigroups 2008. Available from \newline
{\tt http://www.math.uos.de/normaliz}.

\bibitem{cornu-book}{G. Cornu\'ejols, {\it Combinatorial optimization{\rm:}
Packing and covering\/}, CBMS-NSF Regional Conference Series in Applied
Mathematics {\bf 74}, SIAM (2001).}

\bibitem{costa-simis} B. Costa and  A. Simis, Cremona maps defined by
monomials. Preprint, 2011, {\tt arXiv:1101.2413}. 

\bibitem{pan} G. Gonzalez-Sprinberg and I. Pan,
On the monomial birational maps of the projective space,  An. Acad.
Brasil. Ci$\hat{\rm e}$nc.  {\bf 75}  (2003),  no. 2, 129--134.

\bibitem{Kor}{A. B. Korchagin, On birational monomial transformations of plane,
 Int. J. Math. Math. Sci. {\bf 32} (2004),  1671--1677.}

\bibitem{Schr}{A. Schrijver, {\it Theory of Linear and Integer
Programming\/}, John Wiley \& Sons, New York, 1986.}


\bibitem{bir2003}{A. Simis, Cremona transformations and some related
algebras, J. Algebra {\bf 280} (2004), no. 1, 162--179.}

\bibitem{SiVi} {A. Simis and
R. H.  Villarreal, Constraints for the normality of monomial subrings and
birationality, Proc. Amer. Math. Soc. {\bf 131} (2003), no. 7, 2043--2048.}

\bibitem{birational-linear}  A. Simis and R. H. Villarreal,
Linear syzygies and birational combinatorics, Results Math.
{\bf 48} (2005), no. 3-4, 326--343.

\end{thebibliography}

\end{document}